\documentclass[11pt]{article}
\usepackage{amsmath}
\usepackage{amssymb}
\usepackage{amsthm}
\usepackage{bookmark}
\usepackage{geometry}
\usepackage{hyperref}
\newtheorem{theorem}{Theorem}
\newtheorem{corollary}[theorem]{Corollary}
\newtheorem{proposition}[theorem]{Proposition}
\theoremstyle{definition}
\newtheorem{example}[theorem]{Example}
\DeclareMathOperator{\im}{im}
\renewcommand{\labelenumi}{\upshape(\roman{enumi})}
\renewcommand{\theenumi}\labelenumi

\title{A general approach to permutation polynomials from quadratic forms}
\author{Ruikai Chen\\\small School of Mathematical Sciences, South China Normal University, Guangzhou 510631, China\\\small Email: \href{mailto:chen.rk@outlook.com}{chen.rk@outlook.com}}
\date{}

\begin{document}

\maketitle

\begin{abstract}
We investigate a family of polynomials over finite fields of characteristic $2$ related to quadratic forms, and present a general treatment to study permutation properties of those polynomials. By determining some character sums associated with quadratic forms, we explicitly describe several classes of permutation polynomials.

\textit{Keywords:} character sum, finite field, permutation, polynomial, quadratic form
\end{abstract}

\section{Introduction}

Let $\mathbb F_q$ be a finite field of characteristic $2$, $\mathbb F_{q^n}$ being its extension of degree $n$. A polynomial over $\mathbb F_{q^n}$ naturally defines a map on $\mathbb F_{q^n}$. It is called a permutation polynomial of $\mathbb F_{q^n}$ if the corresponding map is a permutation of $\mathbb F_{q^n}$. Consider polynomials of the form
\[\sum_{i,j}a_{ij}x^{2^i+2^j}\]
with coefficients in $\mathbb F_{q^n}$. Such permutation polynomials have attracted significant attention over the decades. See \cite[Table 1]{yu2024constructing} for a list of recent results. Most permutation polynomials studied in the literature have certain forms with several coefficients. In this paper, we seek to establish a framework for these permutation polynomials, as shown in Theorem \ref{main} below. This begins with a simple observation (see \cite[Theorem 7.7]{lidl1997} for example).

\begin{theorem}\label{permutation}
A polynomial $f$ over $\mathbb F_{q^n}$ is a permutation polynomial of $\mathbb F_{q^n}$ if and only if $\sum_{v\in\mathbb F_{q^n}}\chi(uf(v))=0$ for every $u\in\mathbb F_{q^n}^*$, where $\chi$ is a nontrivial additive character of $\mathbb F_{q^n}$.
\end{theorem}

We will see in the next section that for polynomials in the form mentioned above, this character sum is closely related to quadratic forms over finite fields, which allows us to characterize those permutation polynomials.

The following notation will be used throughout. Consider $\mathbb F_{q^n}$ as an $n$-dimensional vector space over $\mathbb F_q$ and the polynomial ring $\mathbb F_{q^n}[x]$. A polynomial over some extension of $\mathbb F_q$ is called a $q$-linear polynomial if the exponent of each of its terms is a power of $q$. Such a polynomial can be viewed as a linear endomorphism of $\mathbb F_{q^n}/\mathbb F_q$. By abuse of notation, let $x^{q^n}=x$; i.e., polynomials over $\mathbb F_{q^n}$ are identified with maps on $\mathbb F_{q^n}$. For a $2$-linear polynomial $L$ over $\mathbb F_{q^n}$ given by
\[L(x)=\sum_ia_ix^{2^i},\]
define $L^\prime$ as
\[L^\prime(x)=\sum_i(a_ix)^{2^{-i}},\]
where $x^{2^{-1}}$ denotes the inverse of the automorphism $x^2$ of $\mathbb F_{q^n}$. Denote by $\mathrm{Tr}$ and $\mathrm N$ the trace map and the norm map of $\mathbb F_{q^n}/\mathbb F_q$ respectively. Let $\chi$ be the canonical additive character of $\mathbb F_{q^n}$, and $\psi$ the canonical additive character of $\mathbb F_q$.

In general, if $L$ is $q$-linear, then $\mathrm{Tr}(uL(v))=\mathrm{Tr}(L^\prime(u)v)$ for all $u,v\in\mathbb F_{q^n}$ as easily seen. In other words, $\ker L$ is the orthogonal complement of $\im L^\prime$ with respect to the standard trace form of $\mathbb F_{q^n}/\mathbb F_q$. The additive characters depend only on the trace over $\mathbb F_2$, so $\chi(uL(v))=\chi(L^\prime(u)v)$ for all $u,v\in\mathbb F_{q^n}$ whenever $L$ is $2$-linear.

\section{Quadratic forms over finite fields of characteristic $2$}

Let $L$ be a $q$-linear polynomial over $\mathbb F_{q^n}$, and define
\[\mathcal S(L)=\sum_{v\in\mathbb F_{q^n}}\chi(vL(v)),\]
where $\chi(vL(v))=\psi(\mathrm{Tr}(vL(v)))$. Note that $\mathrm{Tr}(xL(x))$ induces a quadratic form of $\mathbb F_{q^n}/\mathbb F_q$. Specifically, given a basis $\beta_1,\dots,\beta_n$ of $\mathbb F_{q^n}/\mathbb F_q$ and $v=v_1\beta_1+\dots+v_n\beta_n$ with $v_1,\dots,v_n\in\mathbb F_q$, one has
\[\mathrm{Tr}(vL(v))=\sum_{i=1}^n\sum_{j=1}^nv_iv_j\mathrm{Tr}(\beta_iL(\beta_j)).\]
According to \cite[Theorem 6.30]{lidl1997}, by a change of basis $\mathrm{Tr}(vL(v))$ can be written as
\[v_1v_2+av_1^2+bv_2^2+v_3v_4+v_5v_6+\dots+v_{2r-1}v_{2r}\]
for some integer $r$ with $0\le 2r\le n$ and some elements $a,b\in\mathbb F_q$, or
\[v_1^2+v_2v_3+v_4v_5+\dots+v_{2r}v_{2r+1}\]
for some integer $r$ with $0<2r+1\le n$. Clearly, $\mathcal S(L)=0$ in the latter case. In general, the following propositions help to calculate the sum.

\begin{proposition}\label{sum}
For $a,b\in\mathbb F_{q^n}$, we have
\[\sum_{v_1\in\mathbb F_q}\sum_{v_2\in\mathbb F_q}\psi(v_1v_2+av_1+bv_2)=\psi(ab)q.\]
\end{proposition}
\begin{proof}
Observe that
\[\sum_{v_1\in\mathbb F_q}\sum_{v_2\in\mathbb F_q}\psi(v_1v_2+av_1+bv_2)=\sum_{v_1\in\mathbb F_q}\psi(av_1)\sum_{v_2\in\mathbb F_q}\psi((v_1+b)v_2),\]
where the inner sum is $q$ if $v_1=b$, and is $0$ otherwise. Hence, the sum is exactly $\psi(ab)q$.
\end{proof}

\begin{proposition}\label{zero}
For a $q$-linear polynomial $L$ over $\mathbb F_{q^n}$, if $\mathrm{Tr}(xL(x))$ vanishes on $\ker(L^\prime+L)$, then $\mathcal S(L)^2=q^n|\ker(L^\prime+L)|$; otherwise, $\mathcal S(L)=0$.
\end{proposition}
\begin{proof}
Note that
\[\begin{split}\mathcal S(L)^2&=\sum_{u\in\mathbb F_{q^n}}\chi(uL(u))\sum_{v\in\mathbb F_{q^n}}\chi((u+v)L(u+v))\\&=\sum_{u\in\mathbb F_{q^n}}\sum_{v\in\mathbb F_{q^n}}\chi(vL(u)+uL(v))\chi(vL(v))\\&=\sum_{v\in\mathbb F_{q^n}}\chi(vL(v))\sum_{u\in\mathbb F_{q^n}}\chi(uL^\prime(v)+uL(v))\\&=q^n\sum_{v\in\ker(L^\prime+L)}\chi(vL(v)).\end{split}\]
For $u,v\in\mathbb F_{q^n}$,
\[\mathrm{Tr}((u+v)L(u+v))=\mathrm{Tr}(uL(u))+\mathrm{Tr}(vL(v))+\mathrm{Tr}(u(L^\prime+L)(v)),\]
so $\mathrm{Tr}(xL(x))$ induces an additive homomorphism from $\ker(L^\prime+L)$ to $\mathbb F_q$. Moreover, the kernel of this homomorphism is clearly a subspace over $\mathbb F_q$, and consequently its image has order either $q$ or $1$. In the former case, $\mathrm{Tr}(xL(x))$ acts as a homomorphism onto $\mathbb F_q$, and thus $\mathcal S(L)=0$. Otherwise, the homomorphism is trivial and $\mathcal S(L)^2=q^n|\ker(L^\prime+L)|$.
\end{proof}

Now we consider some specific polynomials, and, given their coefficients, determine whether $\mathcal S(L)$ is zero.

\begin{theorem}\label{degree2}
Let $n=2$ and $L(x)=ax^q+bx$ with $a,b\in\mathbb F_{q^2}$. Then $\mathcal S(L)=0$ if and only if $a^q+a=0$ and $b\ne0$.
\end{theorem}
\begin{proof}
Note that $(L^\prime+L)(x)=(a^q+a)x^q$. If $a^q+a\ne0$, then $\ker(L^\prime+L)=\{0\}$. Otherwise, $\ker(L^\prime+L)=\mathbb F_{q^n}$ and
\[\mathrm{Tr}\big(ax^{q+1}+bx^2\big)=\mathrm{Tr}(a)x^{q+1}+\mathrm{Tr}\big(bx^2\big)=\mathrm{Tr}\big(bx^2\big).\]
The result follows immediately from Proposition \ref{zero}.
\end{proof}

Furthermore, $\mathrm{Tr}\big(ax^{q+1}+bx^2\big)$ vanishes on $\mathbb F_{q^2}$ if and only if $a^q+a=b=0$. Indeed, if it is the zero quadratic form of $\mathbb F_{q^2}/\mathbb F_q$, then $\mathcal S(L)=q^2$ and $\ker(L^\prime+L)=\mathbb F_{q^2}$ by Proposition \ref{zero}, which implies $a^q+a=0$.

Next, consider a binomial $L(x)=ax^{q^k}+bx$ with a positive integer $k$ for general $n$. Without loss of generality, assume $\gcd(k,n)=1$ (replace $q^{\gcd(k,n)}$ by $q$ if $\gcd(k,n)\ne1$). The following result can be derived from \cite{coulter1999evaluation}. Here, we give a straightforward treatment with explicit conditions stated directly on the coefficients.

\begin{theorem}\label{binomial}
Let $k$ be an integer satisfying $0<2k<n$ and $\gcd(k,n)=1$, and $L(x)=ax^{q^k}+bx$ with $a,b\in\mathbb F_{q^n}$. Then $\mathcal S(L)=0$ if and only if
\begin{itemize}
\item $a=0$ and $b\ne0$, or
\item $n$ is odd, $a\ne0$ and $\mathrm{Tr}\Big(ba^{-(q^{kn}+1)/(q^k+1)}\Big)\ne1$, or
\item $n$ is even, $a^{(q^n-1)/(q+1)}=1$ and $\sum_{i=0}^{n/2-1}b^{q^{2ki}}a^{-2(q^{2ki}-1)/(q^k+1)}\ne0$.
\end{itemize}
\end{theorem}
\begin{proof}
The case $a=0$ is trivial, so we assume $a$ is nonzero. An element $\alpha$ of $\mathbb F_{q^n}^*$ is a root of $L^\prime+L$ if and only if
\[a\alpha+a^{q^k}\alpha^{q^{2k}}=0;\]
that is,
\[\alpha^{q^{2k}-1}=a^{1-q^k}.\]
Suppose first that $n$ is odd. Then $\gcd\big(q^{2k}-1,q^n-1\big)=q-1$ and $q^{kn}+1$ is divisible by $q^k+1$. Let $\alpha=a^{-\frac{q^n}2(q^{kn}+1)/(q^k+1)}$, so that
\[\alpha^{q^k+1}=a^{-\frac{q^n}2(q^{kn}+1)}=a^{-\frac{q^n}2\cdot2}=a^{-1}.\]
Hence, $(L^\prime+L)(\alpha)=0$ and $\ker(L^\prime+L)=\alpha\mathbb F_q$, while
\[\mathrm{Tr}\Big(a\alpha^{q^k+1}x^{q^k+1}+b\alpha^2x^2\Big)\equiv\mathrm{Tr}\Big(1+ba^{-(q^{kn}+1)/(q^k+1)}\Big)x^2\mod{x^q-x}.\]
Therefore, $\mathrm{Tr}(xL(x))$ vanishes on $\ker(L^\prime+L)$ if and only if $\mathrm{Tr}\Big(ba^{-(q^{kn}+1)/(q^k+1)}\Big)=1$.

Consider the case of even $n$. Now $\gcd\big(q^{2k}-1,q^n-1\big)=q^2-1$ while $\gcd\big(q^k-1,q^n-1\big)=q-1$, and there exists $\alpha\in\mathbb F_{q^n}^*$ such that $(L^\prime+L)(\alpha)=0$ if and only if $a^{(1-q^k)(q^n-1)/(q^2-1)}=1$; i.e., $a^{(q^n-1)/(q+1)}=1$. Suppose that this is the case. Since $\gcd\big(q^k+1,q^n-1\big)=q+1$, one can choose $\alpha\in\mathbb F_{q^n}$ with $\alpha^{q^k+1}=a^{-1}$, so that $\ker(L^\prime+L)=\alpha\mathbb F_{q^2}$. Let $\mathrm{Tr}_2$ denote the trace map of $\mathbb F_{q^n}/\mathbb F_{q^2}$. Then
\[\mathrm{Tr}_2\Big(a\alpha^{q^k+1}x^{q^k+1}+b\alpha^2x^2\Big)\equiv\mathrm{Tr}_2\Big(a\alpha^{q^k+1}\Big)x^{q+1}+\mathrm{Tr}_2\big(b\alpha^2\big)x^2\mod{x^{q^2}-x}.\]
Taking its trace from $\mathbb F_{q^2}$ to $\mathbb F_q$, we get $\mathrm{Tr}(\alpha xL(\alpha x))$ modulo $x^{q^2}-x$, regarded as a quadratic form of $\mathbb F_{q^2}/\mathbb F_q$. We have just seen that $\mathrm{Tr}(\alpha xL(\alpha x))$ vanishes on $\mathbb F_{q^2}$ if and only if $\mathrm{Tr}_2\Big(a\alpha^{q^k+1}\Big)^q+\mathrm{Tr}_2\Big(a\alpha^{q^k+1}\Big)=\mathrm{Tr}_2\big(b\alpha^2\big)=0$. In fact, $a\alpha^{q^k+1}=1$ and
\[\mathrm{Tr}_2\big(b\alpha^2\big)=\sum_{i=0}^{n/2-1}\big(b\alpha^2\big)^{q^{2ki}}=\alpha^2\sum_{i=0}^{n/2-1}b^{q^{2ki}}\alpha^{2(q^{2ki}-1)}=\alpha^2\sum_{i=0}^{n/2-1}b^{q^{2ki}}a^{-2(q^{2ki}-1)/(q^k+1)},\]
since $2k$ has additive order $\frac n2$ modulo $n$. These, combined with Proposition \ref{zero}, give rise to the result.
\end{proof}

\section{Instances of permutation polynomials}

We now provide a general principle for the permutation behavior of the polynomials being studied.

\begin{theorem}\label{main}
The polynomial
\[f(x)=\sum_{i=0}^{n-1}L_i\Big(x^{q^i+1}\Big),\]
where each $L_i$ is a $2$-linear polynomial over $\mathbb F_{q^n}$, is a permutation polynomial of $\mathbb F_{q^n}$ if and only if $\mathcal S(\ell_u)=0$ for every $u\in\mathbb F_{q^n}^*$, where $\ell_u(x)=\sum_{i=0}^{n-1}L_i^\prime(u)x^{q^i}$.
\end{theorem}
\begin{proof}
For $u\in\mathbb F_{q^n}$, we have
\[\sum_{v\in\mathbb F_{q^n}}\chi(uf(v))=\sum_{v\in\mathbb F_{q^n}}\chi\Bigg(u\sum_{i=0}^{n-1}L_i\Big(v^{q^i+1}\Big)\Bigg)=\sum_{v\in\mathbb F_{q^n}}\chi\Bigg(\sum_{i=0}^{n-1}L_i^\prime(u)v^{q^i+1}\Bigg).\]
The result then follows from Theorem \ref{permutation}.
\end{proof}

We are prepared to give permutation polynomials explicitly, beginning with a construction that yields a wide class of examples.

\begin{proposition}
Let $L$ be a $q$-linear polynomial over $\mathbb F_{q^n}$. Suppose that there is a nonzero element $\alpha$ in $\ker(L^\prime+L)$ with $\mathrm{Tr}(\alpha L(\alpha))\ne1$, and let $\beta=\frac{\alpha^2}{1+\mathrm{Tr}(\alpha L(\alpha))}$. Then $x^2+\beta\mathrm{Tr}(xL(x))$ is a permutation polynomial of $\mathbb F_{q^n}$.
\end{proposition}
\begin{proof}
For $u\in\mathbb F_{q^n}$, let $\ell_u(x)=ux+\mathrm{Tr}(\beta u)L(x)$, so that
\[\mathrm{Tr}\big(u\big(x^2+\beta\mathrm{Tr}(xL(x))\big)\big)=\mathrm{Tr}(x\ell_u(x)),\]
as in Theorem \ref{main}. If $u\ne0$ and $\mathrm{Tr}(\beta u)=0$, then $\ell_u(x)=ux$ and clearly $\mathrm{Tr}(x\ell_u(x))$ does not vanish on $\ker(\ell_u^\prime+\ell_u)$. If $\mathrm{Tr}(\beta u)\ne0$, then $\ker(\ell_u^\prime+\ell_u)=\ker(L^\prime+L)$ as
\[(\ell_u^\prime+\ell_u)(x)=L^\prime(\mathrm{Tr}(\beta u)x)+\mathrm{Tr}(\beta u)L(x)=\mathrm{Tr}(\beta u)(L^\prime+L)(x);\]
furthermore,
\[\mathrm{Tr}(\alpha\ell_u(\alpha))=\mathrm{Tr}\big(\alpha^2u\big)+\mathrm{Tr}(\beta u)\mathrm{Tr}(\alpha L(\alpha))=\mathrm{Tr}(\beta u)\ne0,\]
which means $\mathrm{Tr}(x\ell_u(x))$ does not vanish on $\ker(\ell_u^\prime+\ell_u)$ as well. Hence, $x^2+\beta\mathrm{Tr}(xL(x))$ is indeed a permutation polynomial of $\mathbb F_{q^n}$ by Proposition \ref{zero}.
\end{proof}

In the proof of Proposition \ref{zero}, we have seen that $\mathrm{Tr}(xL(x))$ induces an additive homomorphism from $\ker(L^\prime+L)$ to $\mathbb F_q$. Therefore, such an element $\alpha$ in the preceding proposition exists provided that $|\ker(L^\prime+L)|>2$ or $|\ker(L^\prime+L)|=2$ with $\mathrm{Tr}(xL(x))$ vanishing on $\ker(L^\prime+L)$.

\begin{example}
Let $L(x)=x^{q^k}+x^{q^l}$ for some nonnegative integers $k$ and $l$. Note that $(L^\prime+L)(1)=0$ and $\mathrm{Tr}(L(1))=0\ne1$, so $x^2+\mathrm{Tr}(xL(x))$ is a permutation polynomial of $\mathbb F_{q^n}$.
\end{example}

\begin{theorem}
Let $n=2$, $L_0$ and $L_1$ be $2$-linear polynomials over $\mathbb F_{q^n}$. Then $L_1\big(x^{q+1}\big)+L_0\big(x^2\big)$ is a permutation polynomial of $\mathbb F_{q^2}$ if and only if $L_1(x^q+x)=0$ and $\ker L_0=\{0\}$.
\end{theorem}
\begin{proof}
As a result of Theorem \ref{degree2}, $L_1\big(x^{q+1}\big)+L_0\big(x^2\big)$ is a permutation polynomial of $\mathbb F_{q^2}$ if and only if $L_1^\prime(u)^q+L_1^\prime(u)=0$ and $L_0^\prime(u)\ne0$ for every $u\in\mathbb F_{q^2}^*$. Note that
\[\mathrm{Tr}(uL_1(x^q+x))=\mathrm{Tr}(L_1^\prime(u)(x^q+x))=\mathrm{Tr}((L_1^\prime(u)^q+L_1^\prime(u))x)\]
for every $u\in\mathbb F_{q^2}$. Accordingly, $\im L_1^\prime\subseteq\mathbb F_q$ if and only if $L_1(x^q+x)=0$. Furthermore, $\ker L_0^\prime=\{0\}$ if and only if $\ker L_0=\{0\}$ as $\dim_{\mathbb F_q}\ker L_0+\dim_{\mathbb F_q}\im L_0^\prime=n$.
\end{proof}

Next, consider permutation polynomials induced by the binomial $ax^{q^k}+bx$ studied in Theorem \ref{binomial}.

\begin{example}
Given a $2$-linear polynomial $L_0$ over $\mathbb F_{q^n}$, the condition for $L_0\big(x^2\big)+\mathrm{Tr}\big(x^{q+1}\big)$ to be a permutation polynomial is immediately derived from the above results with Theorem \ref{binomial}, although it has been characterized in \cite[Theorem 2, Theorem 14]{chen2025permutation}. Such permutation polynomials are discussed explicitly there.
\end{example}

\begin{example}
Let $n=3$, $L_0(x)=ax^{2^{-1}}$ and $L_1(x)=x^{q^2}+x$. Since $\ker L_0^\prime=\{0\}$, the polynomial $L_0\big(x^2\big)+L_1\big(x^{q+1}\big)$ is a permutation polynomial of $\mathbb F_{q^3}$ if and only if $\mathrm{Tr}\Big(L_0'(u)L_1'(u)^{-(q^3+1)/(q+1)}\Big)\ne1$ for every $u\in\mathbb F_{q^n}$ such that $L_1'(u)\ne0$, by Theorem \ref{binomial} and Theorem \ref{main}; that is, $\mathrm{Tr}\Big((au)^2(u^q+u)^{-(q^3+1)/(q+1)}\Big)\ne1$ for every $u\in\mathbb F_{q^n}$ such that $u^q+u\ne0$. If $a\in\mathbb F_q^*$ and $u^q+u\ne0$, then
\[\begin{split}&\mathrel{\phantom{=}}\mathrm{Tr}\Big((au)^2(u^q+u)^{-(q^3+1)/(q+1)}\Big)\\&=\mathrm{Tr}\Big(a^2u^2(u^q+u)^{-q^2+q-1}\Big)\\&=\mathrm{Tr}\big(a^2u^2(u^q+u)^{2q}\big)\mathrm N(u^q+u)^{-1}\\&=\mathrm{Tr}\big(a^2u^{2q^2+2}+a^2u^{2q+2}\big)\mathrm N(u^q+u)^{-1}\\&=0.\end{split}\]
This shows that $x^{q^2+1}+x^{q+1}+ax$ is a permutation polynomial of $\mathbb F_{q^3}$ for $a\in\mathbb F_q^*$, as in \cite[Theorem 4]{tu2014several}, from another perspective.
\end{example}

\begin{theorem}
Let $n$ be odd, $k$ be an integer satisfying $0<2k<n$ and $\gcd(k,n)=1$, and $L_0$ be a $2$-linear polynomial over $\mathbb F_{q^n}$. Then $x^{q^k+1}+L_0\big(x^2\big)$ is a permutation polynomial of $\mathbb F_{q^n}$ if and only if $\mathrm{Tr}\Big(L_0^\prime\big(u^{q^k+1}\big)u^{-2}\Big)\ne1$ for every $u\in\mathbb F_{q^n}^*$.
\end{theorem}
\begin{proof}
By Theorem \ref{binomial}, $x^{q^k+1}+L_0\big(x^2\big)$ is a permutation polynomial of $\mathbb F_{q^n}$ if and only if
\[\mathrm{Tr}\Big(L_0^\prime(u)u^{-(q^{kn}+1)/(q^k+1)}\Big)\ne1\]
for every $u\in\mathbb F_{q^n}^*$, where substituting $u^{q^k+1}$ for $u$ on the left side gives
\[\mathrm{Tr}\Big(L_0^\prime\Big(u^{q^k+1}\Big)u^{-q^{kn}-1}\Big)=\mathrm{Tr}\Big(L_0^\prime\Big(u^{q^k+1}\Big)u^{-2}\Big).\]
The result follows since $\gcd\big(q^k+1,q^n-1\big)=1$.
\end{proof}

\begin{example}
Let $n=3$ and consider $x^{q+1}+ax^{2q}+bx^2$ for $a,b\in\mathbb F_{q^3}$. It is a permutation polynomial of $\mathbb F_{q^3}$ if and only if $\mathrm{Tr}\big(L_0^\prime(u^{q+1})u^{-2}\big)\ne1$ for every $u\in\mathbb F_{q^3}^*$, where $L_0(x)=ax^q+bx$ with
\[\mathrm{Tr}\big(L_0^\prime\big(u^{q+1}\big)u^{-2}\big)=\mathrm{Tr}\big(a^{q^2}u^{q^2-1}+bu^{q-1}\big)=\mathrm{Tr}\big(au^{1-q}+bu^{q-1}\big).\]
When $a\ne0$, there exists $u\in\mathbb F_{q^3}^*$ with $\mathrm{Tr}\big(au^{1-q}\big)=1$ by \cite[Theorem 5.3]{moisio2008kloosterman}. When $ab\ne0$, $\mathrm{Tr}\big(au^{1-q}+bu^{q-1}\big)\ne1$ for every $u\in\mathbb F_{q^3}^*$ if and only if $\mathrm N(a)+\mathrm N(b)=ab$ by \cite[Lemma 8]{chen2024characterizations}. Combining these results we conclude that $x^{q+1}+ax^{2q}+bx^2$ is a permutation polynomial of $\mathbb F_{q^3}$ if and only if $\mathrm N(a)+\mathrm N(b)=ab$.
\end{example}

\begin{example}
Let $n=2k+1$ and $L_0(x)=ax^{q^{n-1}}$ for $a\in\mathbb F_{q^n}$, so that
\[\mathrm{Tr}\Big(L_0^\prime\Big(u^{q^k+1}\Big)u^{-2}\Big)=\mathrm{Tr}\Big(a^qu^{q^{k+1}+q-2}\Big)\]
for $u\in\mathbb F_{q^n}$. Consider the case $q=4$ with $a^{2^n-1}\in\mathbb F_4\setminus\mathbb F_2$, where $a\in\alpha\mathbb F_{2^n}$ for some $\alpha\in\mathbb F_4\setminus\mathbb F_2$ as $\alpha^{2^n-1}=\alpha^{2-1}=\alpha$. Now $\mathrm{Tr}\big(u^{2^n+1}\big)\in\mathbb F_4\cap\mathbb F_{2^n}=\mathbb F_2$ and
\[\mathrm{Tr}\big(\alpha u^{2^n+1}\big)=\alpha\mathrm{Tr}\big(u^{2^n+1}\big)\ne1\]
for every $u\in\mathbb F_{4^n}$, so
\[\mathrm{Tr}\Big(L_0^\prime\Big(u^{q^k+1}\Big)u^{-2}\Big)=\mathrm{Tr}\Big(a^4u^{2^{2k+2}+2}\Big)=\mathrm{Tr}\big(a^2u^{2^n+1}\big)^2\ne1.\]
Therefore, the polynomial $x^{2^n+2}+ax$, as well as $x^{q^k+1}+ax^{2q^{n-1}}$, is a permutation polynomial of $\mathbb F_{4^n}$ (cf. \cite[Corollary 2.3]{zieve2013permutation}).
\end{example}

We turn to another form of polynomials involving the trace map. For a basis $\beta_1,\dots,\beta_n$ of $\mathbb F_{q^n}/\mathbb F_q$, the map from $\mathbb F_{q^n}$ to $\mathbb F_q^n$ defined by
\[v\mapsto(\mathrm{Tr}(\beta_1v),\dots,\mathrm{Tr}(\beta_nv))\qquad(v\in\mathbb F_{q^n})\]
is bijective. This leads to the following results.

\begin{theorem}\label{tr}
Let $L_0$ and $L_1$ be $q$-linear polynomials over $\mathbb F_{q^n}$ with a nonnegative integer $l$. Then $L_0\Big(x^{2^l}\Big)+L_1(x)\mathrm{Tr}(x)$ is a permutation polynomial of $\mathbb F_{q^n}$ if and only if, for every $u\in\mathbb F_{q^n}^*$,
\begin{itemize}
\item $L_1^\prime(u)\in\mathbb F_q$ and $L_0^\prime(u)^2+L_1^\prime(u)^{2^l}\ne0$, or
\item $1,L_0^\prime(u),L_1^\prime(u)^{2^l}$ are linearly independent over $\mathbb F_q$.
\end{itemize}
\end{theorem}
\begin{proof}
For $u,v\in\mathbb F_{q^n}$, we have
\[\mathrm{Tr}\Big(uL_0\Big(v^{2^l}\Big)+uL_1(v)\mathrm{Tr}(v)\Big)=\mathrm{Tr}\Big(L_0^\prime(u)v^{2^l}\Big)+\mathrm{Tr}(L_1^\prime(u)v)\mathrm{Tr}(v).\]
Suppose that $L_0^\prime(u)=a+bL_1^\prime(u)^{2^l}$ for some $a,b\in\mathbb F_q$. Then
\[\mathrm{Tr}\Big(uL_0\Big(v^{2^l}\Big)+uL_1(v)\mathrm{Tr}(v)\Big)=a\mathrm{Tr}(v)^{2^l}+b\mathrm{Tr}(L_1^\prime(u)v)^{2^l}+\mathrm{Tr}(L_1^\prime(u)v)\mathrm{Tr}(v).\]
If $L_1^\prime(u)\in\mathbb F_q$, then $L_0^\prime(u)\in\mathbb F_q$ and
\[\begin{split}&\mathrel{\phantom{=}}\sum_{v\in\mathbb F_{q^n}}\chi\Big(uL_0\Big(v^{2^l}\Big)+uL_1(v)\mathrm{Tr}(v)\Big)\\&=\sum_{v\in\mathbb F_{q^n}}\psi\Big(L_0^\prime(u)\mathrm{Tr}(v)^{2^l}+L_1^\prime(u)\mathrm{Tr}(v)^2\Big)\\&=\sum_{v\in\mathbb F_{q^n}}\psi\Big(\Big(L_0^\prime(u)^2+L_1^\prime(u)^{2^l}\Big)\mathrm{Tr}(v)^{2^{l+1}}\Big),\end{split}\]
which is zero if and only if $L_0^\prime(u)^2+L_1^\prime(u)^{2^l}\ne0$. If $L_1^\prime(u)\notin\mathbb F_q$, then
\[\begin{split}&\mathrel{\phantom{=}}\sum_{v\in\mathbb F_{q^n}}\chi\Big(uL_0\Big(v^{2^l}\Big)+uL_1(v)\mathrm{Tr}(v)\Big)\\&=\sum_{v\in\mathbb F_{q^n}}\psi\Big(a^{2^{-l}}\mathrm{Tr}(v)+b^{2^{-l}}\mathrm{Tr}(L_1^\prime(u)v)+\mathrm{Tr}(L_1^\prime(u)v)\mathrm{Tr}(v)\Big),\end{split}\]
which is always nonzero by Proposition \ref{sum}.

If $L_0^\prime(u)$ is not a linear combination of $1$ and $L_1^\prime(u)^{2^l}$ over $\mathbb F_q$, then
\[\begin{split}&\mathrel{\phantom{=}}\sum_{v\in\mathbb F_{q^n}}\chi\Big(uL_0\Big(v^{2^l}\Big)+uL_1(v)\mathrm{Tr}(v)\Big)\\&=\sum_{v\in\mathbb F_{q^n}}\psi\Big(\mathrm{Tr}\Big(L_0^\prime(u)v^{2^l}\Big)+\mathrm{Tr}\Big(L_1^\prime(u)^{2^l}v^{2^l}\Big)\mathrm{Tr}(v)^{2^l}\Big)\\&=\sum_{v\in\mathbb F_{q^n}}\psi(\mathrm{Tr}(L_0^\prime(u)v))\psi\Big(\mathrm{Tr}\Big(L_1^\prime(u)^{2^l}v\Big)\mathrm{Tr}(v)\Big),\end{split}\]
where, letting $v_0=\mathrm{Tr}(L_0^\prime(u)v)$, the sum is divisible by
\[\sum_{v_0\in\mathbb F_q}\psi(v_0)=0.\]
A simple investigation on the conditions completes the proof.
\end{proof}

\begin{example}
Let $k$ be an integer satisfying $0<k<n$ and $\gcd(k,n)=1$. We claim that, for $a\in\mathbb F_{q^n}^*$, the elements $1,u,au^{q^k}+au$ are linearly independent over $\mathbb F_q$ for every $u\in\mathbb F_{q^n}\setminus\mathbb F_q$ if and only if $\mathrm{Tr}\big(a^{-1}\big)\ne0$ and $\mathrm N(a+c)\ne\mathrm N(a)$ for every $c\in\mathbb F_q^*$. As a consequence, the polynomial $(ax)^{q^{n-k}}+ax+x\mathrm{Tr}(x)$ is a permutation polynomial of $\mathbb F_{q^n}$ if and only if $a$ satisfies these conditions. To see this, let $L_0(x)=(ax)^{q^{n-k}}+ax$ and $L_1(x)=x$ as in Theorem \ref{tr}, where $L_0^\prime(x)=ax^{q^k}+ax$ and $L_0^\prime(u)^2+L_1^\prime(u)=u\ne0$ for every $u\in\mathbb F_q^*$.

For the claim, consider the polynomial $\lambda_c(x)=ax^{q^k}+ax+cx$ for $c\in\mathbb F_q$. Assume that $1,u,au^{q^k}+au$ are linearly independent over $\mathbb F_q$ for every $u\in\mathbb F_{q^n}\setminus\mathbb F_q$, which implies $\lambda_c(u)\notin\mathbb F_q$. Then $\ker\lambda_c\subseteq\mathbb F_q$; that is, for $\alpha\in\mathbb F_{q^n}$, if $\alpha^{q^k-1}=a^{-1}(a+c)$, then $\alpha\in\mathbb F_q$. As $\gcd(q^k-1,q^n-1)=q-1$, it follows that $\mathrm N\big(a^{-1}(a+c)\big)=1$ implies $a^{-1}(a+c)=1$. When $c=0$, since $\lambda_c$ vanishes on $\mathbb F_q$, we have $\im\lambda_c\cap\mathbb F_q=\{0\}$ by the assumption, where $\im\lambda_c=a\ker\mathrm{Tr}$. A straightforward argument shows that $\mathrm{Tr}\big(a^{-1}\big)\ne0$. This proves the necessity.

Assume the converse. If $c=0$, then clearly $\ker\lambda_c=\mathbb F_q$ and $\im\lambda_c\cap\mathbb F_q=\{0\}$. If $c\in\mathbb F_q^*$, then $\mathrm N(a+c)\ne\mathrm N(a)$ and $\lambda_c$ permutes $\mathbb F_{q^n}$, while $\lambda_c$ maps $\mathbb F_q$ into $\mathbb F_q$. In either case, $\lambda_c(u)\notin\mathbb F_q$ for every $u\in\mathbb F_{q^n}\setminus\mathbb F_q$. This proves the claim.
\end{example}

\begin{example}
For $a\in\mathbb F_{q^n}^*$ with an integer $k$ satisfying $0<k<n$ and $\gcd(k,n)=1$, the polynomial $ax^{q^k}+ax+x\mathrm{Tr}(x)$ is a permutation polynomial of $\mathbb F_{q^n}$ if and only if $a\in\mathbb F_q^*$ and $n$ is odd with $\gcd(n,q-1)=1$. Let $L_0(x)=ax^{q^k}+ax$ and $L_1(x)=x$. If $L_0(x)+L_1(x)\mathrm{Tr}(x)$ is a permutation polynomial of $\mathbb F_{q^n}$, then $a\in\mathbb F_q$ since the kernel of $L_0^\prime$ is contained in $\mathbb F_q$ according to Theorem \ref{tr}, where $L_0^\prime(x)=(ax)^{q^{n-k}}+ax$. Then necessarily $a\in\mathbb F_q^*$, and thereby $L_0^\prime(u)^2+L_1^\prime(u)=u\ne0$ for every $u\in\mathbb F_q^*$. Moreover, $\mathrm{Tr}\big(a^{-1}\big)=\mathrm{Tr}(1)a^{-1}$ and $\mathrm N\big(a^{-1}(a+c)\big)=\big(1+a^{-1}c\big)^n$ for $c\in\mathbb F_q$. The result follows from the claim in the previous example.
\end{example}

\begin{corollary}
Let $k$ and $l$ be nonnegative integers with $q=2^m$ and $a\in\mathbb F_{q^n}^*$. Then $ax^{2^lq^k}+x\mathrm{Tr}(x)$ is a permutation polynomial of $\mathbb F_{q^n}$ if and only if $n$ is odd, $\gcd\big(2^{l+mk}-1,(q^n-1)/(q-1)\big)=1$ and $a\in\mathbb F_q^*$ with $a^{(q-1)/(2^d-1)}\ne1$, where $d=\gcd(l-1,m)$.
% $\gcd(l,mn)$ divides $m$, $\gcd(n,2^d-1)=1$.
\end{corollary}
\begin{proof}
Let $L_0(x)=ax^{q^k}$ and $L_1(x)=x$. Since $L_0^\prime(x)=(ax)^{q^{n-k}}$ and $L_1^\prime(x)=x$, the polynomial $L_0\Big(x^{2^l}\Big)+L_1(x)\mathrm{Tr}(x)$ is a permutation polynomial of $\mathbb F_{q^n}$ if and only if
\begin{enumerate}
\item\label{c1}$(au)^{2q^{n-k}}+u^{2^l}\ne0$ for every $u\in\mathbb F_q^*$, and
\item\label{c2}$1,u^{2^l},(au)^{q^{n-k}}$ are linearly independent over $\mathbb F_q$ for every $u\in\mathbb F_{q^n}\setminus\mathbb F_q$.
\end{enumerate}

We show that the condition \ref{c2} holds if and only if $a\in\mathbb F_q^*$, $\gcd\big(2^{l+mk}-1,(q^n-1)/(q-1)\big)=1$ and $n$ is odd. Let $\lambda_c(x)=(ax)^{q^{n-k}}+cx^{2^l}$ for $c\in\mathbb F_q$. Clearly, the condition can be rephrased to say that $\lambda_c(u)\notin\mathbb F_q$ for every $u\in\mathbb F_{q^n}\setminus\mathbb F_q$ and every $c\in\mathbb F_q$. Suppose that this is the case. Then $\ker\lambda_c\subseteq\mathbb F_q$. If $a\notin\mathbb F_q$, then $\lambda_c$ maps all nonzero elements of $\mathbb F_q$ to elements outside of $\mathbb F_q$, and thus $\ker\lambda_c=\ker\lambda_c\cap\mathbb F_q=\{0\}$; in this case, $\lambda_c$ maps some element of $\mathbb F_{q^n}\setminus\mathbb F_q$ into $\mathbb F_q$ as $\lambda_c$ permutes $\mathbb F_{q^n}$. The contradiction shows that $a$ belongs to $\mathbb F_q$. For an arbitrary element $\alpha\in\mathbb F_{q^n}^*$, if $\alpha^{(1-2^{l+mk})(q-1)}=1$, then $\alpha^{1-2^{l+mk}}=a^{-1}c$ for some $c\in\mathbb F_q^*$, which means
\[\lambda_c(\alpha)^{q^k}=a\alpha+c\alpha^{2^{l+mk}}=0\]
and then $\alpha\in\mathbb F_q^*$. That is, $\alpha^{(1-2^{l+mk})(q-1)}=1$ implies $\alpha^{q-1}=1$, and thus $\gcd\big(\big(1-2^{l+mk}\big)(q-1),q^n-1\big)$ divides $q-1$. Moreover,
\[\gcd\big(\big(1-2^{l+mk}\big)(q-1),q^n-1\big)=(q-1)\gcd\big(2^{l+mk}-1,(q^n-1)/(q-1)\big),\]
so $\gcd\big(2^{l+mk}-1,(q^n-1)/(q-1)\big)=1$. As $a\in\mathbb F_q$, it is clear that
\[\lambda_c(x)^q+\lambda_c(x)=\lambda_c(x^q+x),\]
where the image of $x^q+x$ on $\mathbb F_{q^n}$ is $\ker\mathrm{Tr}$. If $n$ is even, then $1\in\ker\mathrm{Tr}$ and $u^q+u=1$ for some $u\in\mathbb F_{q^n}\setminus\mathbb F_q$, so that $\lambda_a(u)^q+\lambda_a(u)=\lambda_a(1)=0$. As a consequence, $n$ must be odd.

Conversely, assume $a\in\mathbb F_q^*$, $\gcd\big(2^{l+mk}-1,(q^n-1)/(q-1)\big)=1$ and $n$ is odd. For every $c\in\mathbb F_q$, if $\lambda_c(\alpha)=0$ for some $\alpha\in\mathbb F_{q^n}^*$, then $\alpha^{1-2^{l+mk}}=a^{-1}c\ne0$ and $\alpha^{q-1}=1$ since $\gcd\big(\big(2^{l+mk}-1\big)(q-1),(q^n-1)\big)=q-1$. This implies $\ker\lambda_c\subseteq\mathbb F_q$. For odd $n$, one has
\[\ker\lambda_c\cap\ker\mathrm{Tr}\subseteq\mathbb F_q\cap\ker\mathrm{Tr}=\{0\},\]
and $\lambda_c(u)^q+\lambda_c(u)=0$ only if $u^q+u=0$, as asserted.

Finally, given $a\in\mathbb F_q$, since $(au)^{2q^{n-k}}+u^{2^l}=a^2u^2+u^{2^l}$ for $u\in\mathbb F_q^*$, it is routine to check that the condition \ref{c1} is equivalent to $a^{(q-1)/(2^d-1)}\ne1$.
\end{proof}

\begin{example}
For $a\in\mathbb F_{q^n}$ with $n$ odd, it is stated in \cite[Theorem 5]{blokhuis2001permutations} that $ax^2+x\mathrm{Tr}(x)$ is a permutation polynomial of $\mathbb F_{q^n}$ if $a\in\mathbb F_q^*$ and $a\ne1$. The preceding corollary provides a more general characterization.
\end{example}

\section{Conclusions}

We have determined the character sum $\mathcal S(L)$ for the quadratic form $\mathrm{Tr}(xL(x))$ with some $q$-linear polynomial $L$ over $\mathbb F_{q^n}$, and characterized associated permutation polynomials in Theorem \ref{main}. This outlines a new direction for finding permutation polynomials of those specific forms, and motivates a systematic study towards the classification of quadratic forms defined by polynomials. Apart from explicit constructions, properties of permutation polynomials within this framework also deserve further investigation.

% \section*{Acknowledgments}

% The author thanks the anonymous reviewers for their comments, which have significantly improved the quality of this paper.

% This work was supported by the National Natural Science Foundation of China under Grant No. 12571003.

% \section*{Conflict of interest}

% The author declares that there is no conflict of interest regarding the publication of this paper.

\bibliographystyle{abbrv}
\bibliography{references}

\end{document}